\documentclass{article}

\usepackage{amssymb}
\usepackage{amsthm}
\usepackage{amsmath}
\newtheorem{theorem}{Theorem}[section]

\newtheorem{cor}[theorem]{Corollary}
\theoremstyle{definition}

\theoremstyle{remark}

    \newcommand{\st}{^{\textstyle {\ast }}}
   \newcommand{\ph}{\mbox{$\varphi$}}
    
    \renewcommand{\phi}{\varphi}
   \newcommand{\Ht}{\mbox{$H^{2}$}}
   
   \newcommand{\Hi}{\mbox{$H^{\infty}$}}
   \newcommand{\D}{\mbox{$\mathbb{D}$}}
   \newcommand{\C}{\mbox{$C_{\varphi}$}}
   \newcommand{\Cs}{\mbox{$C_{\varphi}\st$}}

   \newcommand{\W}{\mbox{$W_{\psi,\varphi}$}}
   \newcommand{\Ws}{\mbox{$W_{\psi,\varphi}^{\textstyle\ast}$}}
 	\newfont{\caps}{cmcsc9}  % for Authors
 	\newfont{\jour}{cmti9}  % for Journal Titles
\begin{document}



\title{Complex Symmetric Composition Operators on $H^2$}


\author{Sivaram K. Narayan, Daniel Sievewright, Derek Thompson}

\maketitle

\begin{abstract}
In this paper, we find complex symmetric composition operators on the classical Hardy space $H^2$ whose symbols are linear-fractional but not automorphic. In doing so, we answer a recent question of Noor, and partially answer the original problem posed by Garcia and Hammond. 
\end{abstract}



%%

\section{Introduction}

In this paper, we are interested in composition operators on the classical \textit{Hardy space} \Ht. This Hilbert space is given by the set of analytic functions $\displaystyle f = \sum_{n=0}^{\infty}a_n z^n$ on the open unit disk \D\ such that $$\|f\|^{2}=\sum_{n=0}^{\infty}|a_n|^{2}<\infty.$$ A \textit{composition operator} $C_{\ph}$ on \Ht\ is given by $C_{\ph}f = f \circ \ph$. When \ph\ is an analytic self-map of \D, \C\ is necessarily bounded. A \textit{Toeplitz operator} $T_\psi$ on $H^2$ is given by $T_\psi f = \psi f$ when $\psi \in \Hi$, the space of bounded analytic functions on \D. We occasionally write $\W = T_\psi \C$ and call such an operator a \textit{weighted composition operator}. When \ph\ is linear-fractional, \Cs\ has a simple form which involves composition operators and (adjoints of) Toeplitz operators, which we will put to use later in this paper.

The weak normality properties of composition operators and weighted composition operators on \Ht\ have been of great interest. For example, work has been done on when \C\ and \W\ are normal \cite{Bourdon}, subnormal \cite{ck}, and cohyponormal \cite{coko}. In \cite{Hammond}, Garcia and Hammond posed the following problem: ``Characterize all complex symmetric composition operators \C\ on the classical Hardy space \Ht.'' An operator $T$ on a Hilbert space $\mathcal{H}$ is complex symmetric if there exists a conjugate-linear, isometric involution $J$ on $\mathcal{H}$ so that $T = J T^* J$. We call such an operator $J$ a \textit{conjugation}. For applications and further information about complex symmetric operators, see \cite{Garcia1}, \cite{Garcia2} and \cite{GW}.

Due to a result of Garcia and Wogen \cite{GW}, any automorphism of order $2$ is complex symmetric. In \cite{Bourdon1}, Bourdon and Noor found that other automorphisms cannot be complex symmetric except possibly for the unsolved case when \ph\ is an elliptic automorphism of order 3. In \cite{Noor1}, Noor found the exact conjugation needed for the involutive disk automorphisms to induce complex symmetric composition operators. Jung, Kim, Ko, and Lee were thought to have found a non-automorphic example in \cite{jkkl}, but were disproved by Noor in \cite{Noor2}. Noor stated that even a narrower question remains open: does there exist a non-constant and non-automorphic symbol \ph\ for which \C\ is complex symmetric but not normal on \Ht?

In this paper, we give examples of linear-fractional, non-automorphic maps \ph\ that induce complex symmetric composition operators on $H^2$. In doing so, we answer positively the question of Noor, and we answer in part the question originally posed by Garcia and Hammond. We summarize our results with the following statement.

\begin{theorem}[Summary] Let $a, b \in \D$ with $|a|+|b(1-a)| \leq 1$.  Then for $\sigma(z) = az + b(1-a)$ and $\ph(z) = az/(1-b(1-a)z)$, $C_{\sigma}$ and \C\ are complex symmetric on \Ht. \end{theorem}

Our strategy will be to first prove the result for linear maps $\sigma$ with real fixed point $b$, and from there, we will extend to complex fixed points and then to functions of the form found in \ph\ above. To that end, here we mention the Cowen adjoint formula for composition operators with linear-fractional symbol \cite[Theorem 9.2]{cm1}. The theorem is well-known but we restate it here due to its frequent use throughout the paper.

\begin{theorem}[Cowen adjoint formula]\label{adjoint} If $\sigma = \frac{az+b}{cz+d}$ is an analytic self-map of $\D$, then on $H^2$, $C_{\sigma}^{*} = T_g C_{\ph} T_{h}^{*}$, where $g = 1/(-\overline{b}z+\overline{d}), h = cz + d$, and $\ph = (\overline{a}z-\overline{c})/(-\overline{b}z+\overline{d})$, and $g, h$ necessarily belong to \Hi. The function \ph\ is called the \textit{Krein adjoint} of $\sigma$, while $g, h$ are called the \textit{Cowen auxillary functions} of $\sigma$.\end{theorem}

\section{Results}

Let $J$ be the operator defined by $(Jf)(z) = \overline{f(\overline{z})}$. Then $J$ is a conjugation. We will consider when $J$ commutes with other involutions, so that the involution in our construction is conjugate-linear. 

\begin{theorem}\label{conjlinear} Suppose $J_1$ and $J_2$ are isometric involutions which commute, and exactly one of them is conjugate-linear. Then $J_1 J_2$ is a conjugation. \end{theorem}

\begin{proof} A product of isometries is an isometry. Since $J_1 J_2 J_1 J_2 = J_1 J_2 J_2 J_1 = J_1 J_1 = I$, $J_1 J_2$ is an involution. Without loss of generality, say $J_1$ is conjugate-linear. Then $J_1 J_2 (af + bg) = J_1 (a J_2f + bJ_2g) = \overline{a} J_1 J_2f + \overline{b}J_1 J_2g$. Therefore, $J_1 J_2$ is a conjugation. \end{proof}

\begin{theorem}\label{commute} Suppose that $\psi \in H^{\infty}$ and $\ph: \D \rightarrow \D$ both map $(-1,1)$ into itself. Then $J$ commutes with $T_\psi$ and $C_{\ph}$ and therefore also with \W. \end{theorem}
\begin{proof} $JC_{\ph} f = J f(\ph(z)) = \overline{f(\ph(\overline{z}))}$ and $C_{\ph} J f = C_{\ph} \overline{f(\overline{z})} = \overline{f(\overline{\ph(z)})}$. Now, for multiplication operators, $T_\psi J f = \psi(z) \overline{f(\overline{z})}$ and $J T_\psi f = \overline{\psi(\overline{z})}\overline{f(\overline{z})}$. In both situations, the expressions agree on $(-1,1)$. Since they are the same functions on a set with an accumulation point, they equal each other on the entire open disk. \end{proof}

\begin{theorem}\label{unitary} Suppose $b \in (-1,1), \tau = \frac{b- z}{1 - b z}$ and $\zeta = \sqrt{1-b^2} \frac{1}{1-b z}$. Then $T_\zeta C_\tau$ is unitary and self-adjoint; therefore, it is an isometric involution (but not conjugate-linear). Furthermore, $JT_{\zeta}C_{\tau}$ is a conjugation. \end{theorem}
\begin{proof} The fact that $T_{\zeta} C_{\tau}$ is unitary and self-adjoint can be found in both \cite{Bourdon} and \cite{wcomp} (the weight $\zeta$ is $K_{\tau(0)}$ divided by $\|K_{\tau(0)}\|$). By Theorem \ref{commute}, since $b$ is real, $J$ and $T_{\zeta}C_{\tau}$ commute. By Theorem \ref{conjlinear}, then, $JT_{\zeta} C_{\tau}$ is a conjugation. \end{proof} 

\begin{theorem}\label{adjoint2} Let $a \in \D, b \in (-1,1)$ with $|a|+|b(1-a)| \leq 1$, and let $\sigma(z) = az + b(1-a)$. Furthermore, let $\ph(z) = \overline{a}z/(1-b(1-\overline{a})z), \psi(z) = 1/(1-b(1-\overline{a})z)$. Then $C_{\sigma}^{*} = T_{\psi} C_{\ph}$. \end{theorem}
\begin{proof} By the Cowen adjoint formula found in Theorem \ref{adjoint}, $C_{\sigma}^{*} = T_{\psi} C_{\phi}$.  \end{proof}

We are now ready to perform some calculations that show $C_\sigma$ is complex symmetric. We are purposefully letting $b$ belong to $(-1,1)$ so that our choice of involution works; however, we will later remove this requirement. 

\begin{theorem}\label{mix} Let $a \in \D, b \in (-1,1)$ with $|a|+|b(1-a)| \leq 1$ and let $\sigma(z) = az + b(1-a)$. Furthermore, let $\ph(z) = \overline{a}z/(1-b(1-\overline{a})z), \psi(z) = 1/(1-b(1-\overline{a})z)$. Then $C_\sigma = JT_\zeta C_\tau T_\psi C_{\ph} JT_\zeta C_\tau$, where $\tau, \zeta$ are as in Theorem \ref{unitary}. \end{theorem}

\begin{proof} We show this by direct calculation: 

\begin{align*}
JT_\zeta C_\tau T_\psi C_\varphi JT_\zeta C_\tau f(z) &= 
JT_\zeta C_\tau  T_\psi C_\varphi \overline{\zeta(\overline{z})} \;\overline{f(\tau(\overline{z}))} \\ 
&= JT_\zeta C_\tau \psi(z) \overline{\zeta(\overline{\varphi(z)})} \;\overline{f(\tau(\overline{\varphi(z)}))} \\ 
&= J \zeta(z)  \psi(\tau(z)) \overline{\zeta(\overline{\varphi(\tau(z))})} \;\overline{f(\tau(\overline{\varphi(\tau(z))}))} \\
&=  \overline{\zeta(\overline{z})} \;  \overline{\psi(\tau(\overline{z}))} 
\zeta(\overline{\varphi(\tau(\overline{z}))}) \;f(\tau(\overline{\varphi(\tau(\overline{z}))})). 
\end{align*}

First, we check the compositional symbol. For clarity, we will work this out one composition at a time. Recall that we are assuming $b$ is real. 

\begin{align*}
\tau(\overline{z}) &= \frac{b-\overline{z}}{1-b\overline{z}} 
\end{align*}
\begin{align*}
\overline{\ph(\tau(\overline{z}))} &= \overline{\ph\left(\frac{b-\overline{z}}{1-b\overline{z}}\right)}  \\
&= \overline{\left( \frac{\overline{a}\frac{b-\overline{z}}{1-b\overline{z}}}{1-b(1-\overline{a})\frac{b-\overline{z}}{1-b\overline{z}}} \right) } \\
&= \frac{a\frac{b-z}{1-bz}}{1-b(1-a)\frac{b-z}{1-bz}} \\
&= \frac{a(b-z)}{1-bz-b(1-a)(b-z)} \\
&= \frac{ab-az}{1-b^2+ab^2-abz}
\end{align*}
\begin{align*}
\tau(\overline{\ph(\tau(\overline{z}))}) &=  \frac{b-\frac{ab-az}{1-b^2+ab^2-abz}}{1-b\frac{ab-az}{1-b^2+ab^2-abz}}\\
 &= \frac{b-b^3+ab^3-ab^2z-ab+az}{1-b^2+ab^2-abz-ab^2+abz}\\
&= \frac{b-b^3+ab^3-ab^2z-ab+az}{1-b^2} \\
&= \frac{(1-b^2)(az+b(1-a))}{1-b^2} \\
&= az + b(1-a) = \sigma
\end{align*}

Now, we will evaluate each of the three multiplicative factors separately, and see that they simplify to $1$:
\begin{align*}
\overline{\zeta(\overline{z})} &= \frac{\sqrt{1-b^2}}{1-bz} 
\end{align*}
\begin{align*}
\overline{\psi(\tau(\overline{z}))} &= \overline{\psi\left(\frac{b-\overline{z}}{1-b\overline{z}}\right)} \\
&= \overline{\left(\frac{1}{1-b(1-\overline{a})\frac{b-\overline{z}}{1-b\overline{z}}}\right)} \\
&= \overline{\left(\frac{1-b\overline{z}}{1-b\overline{z}-b(1-\overline{a})(b-\overline{z})}\right)} \\
&= \frac{1-bz}{1-bz-b(1-a)(b-z)} \\
&= \frac{1-bz}{1-bz-b(b-ab-z+az)} \\
&= \frac{1-bz}{1-bz-b^2+ab^2+bz-abz} \\
&= \frac{1-bz}{1-b^2+ab^2-abz}
\end{align*}

For the third multiplication factor, recall that we have already simplified $\overline{\ph(\tau(\overline{z}))}$ above:
\begin{align*}
\zeta(\overline{\ph(\tau(\overline{z}))}) &= \zeta\left(\frac{ab-az}{1-b^2+ab^2-abz}\right) \\
&= \sqrt{1-b^2} \frac{1}{1-b\frac{ab-az}{1-b^2+ab^2-abz}} \\
&= \sqrt{1-b^2}\frac{1-b^2+ab^2-abz}{1-b^2+ab^2-abz-ab^2+abz} \\
&= \sqrt{1-b^2}\frac{1-b^2+ab^2-abz}{1-b^2} \\
&= \frac{1-b^2+ab^2-abz}{\sqrt{1-b^2}}
\end{align*}

Putting these together, we have
\begin{align*}
\overline{\zeta(\overline{z})} \;  \overline{\psi(\tau(\overline{z}))} 
\zeta(\overline{\varphi(\tau(\overline{z}))}) &= \\ 
\left(\frac{\sqrt{1-b^2}}{1-bz}\right) \left(\frac{1-bz}{1-b^2+ab^2-abz}\right)\left(\frac{1-b^2+ab^2-abz}{\sqrt{1-b^2}}\right) &= 1 
\end{align*}
and we are finished. 
\end{proof}

\begin{cor}\label{breal} Let $a \in \D, b \in (-1,1)$ with $|a|+|b(1-a)| \leq 1$ and let $\sigma(z) = az + b(1-a)$. Then $C_{\sigma}$ is complex symmetric. \end{cor}
\begin{proof} By Theorem \ref{unitary}, $J T_\zeta C_\tau $ is a conjugation. By Theorem \ref{adjoint2}, $C_{\sigma}^{*} = T_\psi \C$. By Theorem \ref{mix}, $C_\sigma = J T_\zeta C_\tau T_\psi C_{\ph} J T_\zeta C_\tau  = J T_\zeta C_\tau C_{\sigma}^{*} J T_\zeta C_\tau$. \end{proof}

We now use unitarily equivalence to allow $b$ to be complex.

\begin{cor}\label{bdisk} Let $a, b \in \D$ with $|a|+|b(1-a)| \leq 1$ and let $\sigma(z) = az + b(1-a)$.Then $C_{\sigma}$ is complex symmetric.\end{cor}
\begin{proof} By Corollary \ref{breal} , $C_{\sigma_{1}}$ is complex symmetric, where $\sigma_{1}(z) = az + b_1(1-a)$ with $a \in \D, b_1 \in (-1,1)$ and $|a|+|b_{1}(1-a)| \leq 1$. Let $U_{\theta}$ be defined by $U_{\theta} f = C_{e^{i\theta}z} f = f(e^{i\theta}z)$ for some angle $\theta$ and note that $U_{\theta}$ is a unitary operator with inverse $U_{-\theta}$. Then $U_{-\theta} C_{\sigma_{1}} U_{\theta}$ is unitarily equivalent to $ C_{\sigma}$ where $\sigma = az + e^{i\theta}b_{1}(1-a)$, so it is also complex symmetric. Letting $b = b_{1}e^{i\theta}$, since this is true for any angle $\theta$, the conclusion follows. \end{proof}

%\begin{cor}\label{adisk} Let $a, b \in \D$

The work above shows that if $\sigma$ is a linear symbol with interior fixed point, $C_{\sigma}$ is complex symmetric on $H^2$. However, as we will see below, that work also shows that $C_{\ph}$ is complex symmetric, where \ph\ is the Krein adjoint of $\sigma$. 

\begin{cor}\label{phi} Let $a, b \in \D$ with $|a|+|b(1-a)| \leq 1$. Let $\ph(z) = az/(1-b(1-a)z), \psi(z) = 1/(1-b(1-a)z)$. Then \W\ is complex symmetric.\end{cor}
\begin{proof} $\Ws = C_{\sigma}$, where $\sigma(z) = \overline{a}z + \overline{b}(1-\overline{a})$, is complex symmetric by Corollary \ref{bdisk}. Since the adjoint of \W\ is complex symmetric, \W\ is also complex symmetric. \end{proof}

\begin{cor} Let $a, b \in \D$ with $|a|+|b(1-a)| \leq 1$ and let $\ph(z) = az/(1-b(1-a)z)$. Then $C_{\ph}$ is complex symmetric. \end{cor}
\begin{proof} First, recall that $zH^2$ is a reducing subspace for $C_{\ph}$. Let $\psi(z) = 1/(1-b(1-a)z)$. By Corollary \ref{phi}, \W\ is complex symmetric and therefore so is $a\W$. By the unitary operator $T_{z}: \Ht \rightarrow z\Ht$, $a\W$ is unitarily equivalent to $\C|_{z\Ht}$. Therefore, $\C|_{z\Ht}$ is complex symmetric by some conjugation $J_2: zH^2 \rightarrow zH^2$. On $H^2$, $\C = I \oplus \C|_{z\Ht}$ and $I$ is trivially complex symmetric by any conjugation, say $J_1$, on the constants. Then $\C$ is complex symmetric on $H^2$ by the conjugation $J_1 \oplus J_2$.   \end{proof}

\section{Further Questions}

\begin{enumerate}
\item We have shown several operators $C_{\ph}$ to be complex symmetric when $\ph$ is linear-fractional with interior fixed point and no fixed point on the boundary. However, there are other examples, e.g. $\ph = 1/(3-z)$. Are these operators also complex symmetric?
\item Are there composition operators on $H^2$ that are complex symmetric, yet whose symbols are not linear fractional?
\end{enumerate}

\end{document}